\documentclass[11pt]{amsart}

\usepackage{amsmath} 
\usepackage{amsthm} 
\usepackage{amssymb} 

\usepackage{microtype} 
\usepackage{pinlabel} 

\usepackage{MnSymbol} 

\usepackage[scaled=0.9]{sourcecodepro} 

\usepackage{enumitem}
\usepackage{color}
\usepackage{subfig, caption}

\usepackage{wrapfig}
\usepackage{pdflscape}
\usepackage{array}
\usepackage{tikz-cd}
\usepackage{longtable, booktabs}

\usepackage[hidelinks, pagebackref]{hyperref}

\makeatletter
\define@key{href}{font}{#1}
\makeatother
\usepackage{xpatch}
\newcommand\hrefdefaultfont{\ttfamily}
\xpatchcmd\href{\setkeys{href}{#1}}{\setkeys{href}{font=\hrefdefaultfont,#1}}{}{\fail}

\renewcommand*{\backref}[1]{}
\renewcommand*{\backrefalt}[4]{
  \ifcase #1 
  [No citations.]
  \or [#2]
  \else [#2]
  \fi }

\let\originalleft\left
\let\originalright\right
\renewcommand{\left}{\mathopen{}\mathclose\bgroup\originalleft}
\renewcommand{\right}{\aftergroup\egroup\originalright}

\newcommand{\thsup}{{\rm th}}

\newcommand{\calB}{\mathcal{B}}

\newcommand{\MM}{\mathbb{M}}

\newcommand{\cross}{\times}

\newcommand{\bdy}{\partial} 

\input{header_article.tex}

\theoremstyle{plain}
\newtheorem{XXXtheoremQED}[equation]{Theorem} 
  {\pushQED{\qed}\begin{XXXtheoremQED}}
  {\popQED\end{XXXtheoremQED}}

\newcommand{\fakeenv}{} 

{ 
 \renewcommand{\fakeenv}{#2} 
 \theoremstyle{plain} 
 \newtheorem*{\fakeenv}{#1~\ref{#2}} 
 \begin{\fakeenv}
}
{
 \end{\fakeenv}
}

\newenvironment{restated}[2]  
{ 
 \renewcommand{\fakeenv}{#2} 
 \theoremstyle{definition} 
 \newtheorem*{\fakeenv}{#1~\ref{#2}} 
 \begin{\fakeenv}
}
{
 \end{\fakeenv}
}

\begin{document}

\title{Large volume fibred knots of fixed genus}

\author[K.~Baker]{Kenneth L. Baker}
\author[D.~Futer]{David Futer}
\author[J.~Purcell]{Jessica S. Purcell}
\author[S.~Schleimer]{Saul Schleimer}

\dedicatory{For Alan Reid}

\begin{abstract}
We show that, for hyperbolic fibred knots in the three-sphere, the volume and the genus are unrelated.  
Furthermore, for such knots, the volume is unrelated to strong quasipositivity and Seifert form.
\end{abstract}

\thanks{This work is in the public domain.}

\date{2023-08-15}




\maketitle

\setcounter{section}{1}

\begin{wrapfigure}[24]{r}{1.9in}
\vspace{-5pt}
\centering
\labellist
\small\hair 2pt
\pinlabel {$\Pi$} [l] at 172 453
\pinlabel {$\Phi^3$} [l] at 172 317
\pinlabel {$\sigma_1^{-1}$} [l] at 172 210
\pinlabel {$\Phi^{-3}$} [l] at 172 104
\pinlabel {\footnotesize$\gamma_n$} [bl] at 22 320
\pinlabel {\footnotesize$\delta_n$} [bl] at 22 106
\pinlabel {\footnotesize$\omega_n$} [bl] at -17 21
\endlabellist
\vspace{-0.5cm}
\includegraphics[width = 0.27\textwidth]{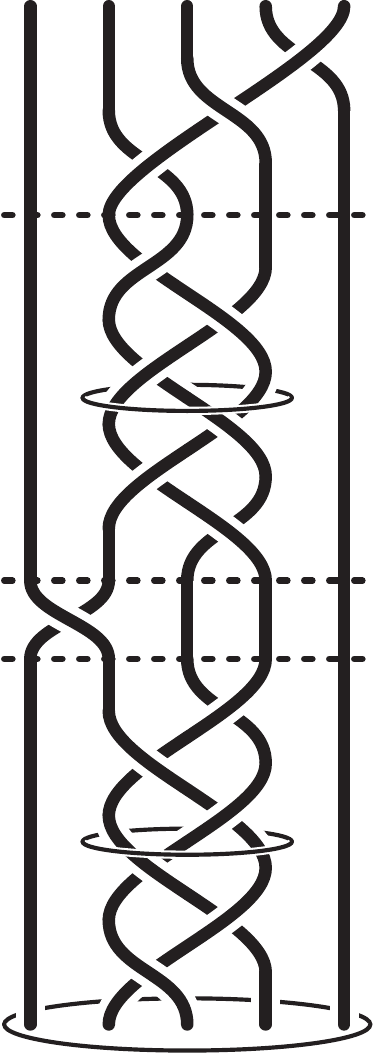}
\caption{}
\label{Fig:Braid}
\end{wrapfigure}

In this note we prove the following.

\begin{theorem}
\label{Thm:BigKnots}
For every $g > 1$ and every $V > 0$ there is a knot $w \subset S^3$ with the following properties. 
\begin{enumerate}
\item
\label{Itm:Knot}
$S^3 - w$ is fibred over the circle, with fibre of genus $g$. 
\item
\label{Itm:Hyp}
$S^3 - w$ is hyperbolic, with volume at least $V$.
\item
\label{Itm:Long}
Longitudinal surgery on $w$ is hyperbolic, with volume at least $V$.
\end{enumerate}
\end{theorem}


\begin{remark}
In the opposite direction, the two-bridge knots with continued fraction $[2g - 1, 1, 2]$ are hyperbolic and fibred of genus $g$, with bounded volume as $g \to \infty$.  
\end{remark}

\refthm{BigKnots} answers a question posed by Reid.
Our proof also provides a family of fibred hyperbolic knots in $S^2 \times S^1$, of fixed genus, whose double branched covers have unbounded volume.
This answers a special case of a question posed by Hirose, Kalfagianni, and Kin~\cite[Question~4]{HiroseKalfagianniKin22}. 
For more details, see \refrem{Other}.

\begin{proof}[Proof of \refthm{BigKnots}]
Fix $g > 1$.  
Let $\calB_{2g + 1}$ be the braid group on $2g+1$ strands. 
Let $\sigma_i$ be the positive half-twist between the $i^\thsup$ and $(i+1)^\thsup$ strands.
We define the following braids.
\begin{align*}
\Pi     &= \sigma_{2g} \cdot \sigma_{2g - 1}^{-1} \cdot \sigma_{2g - 2} \cdot \sigma_{2g - 3}^{-1} 
           \cdots \sigma_4 \cdot \sigma_3^{-1} \cdot \sigma_{2} \\ 
\Phi    &= \sigma_{2} \cdot \sigma_{3}^{-1} \\ 
\beta_n &= \Pi \cdot \Phi^n \cdot \sigma_1^{-1} \cdot \Phi^{-n} 
\end{align*}
\noindent
See \reffig{Braid}, where we take $g = 2$ and $n = 3$. 

Let $\hat{\beta}_n$ be the braid closure of $\beta_n$, taken in $S^3$. 
Let $\omega_n$ be its augmenting braid axis.
Let $\Lambda_n = \left( \hat{\beta}_n \cup \omega_n \right)$ be the resulting two-component link.
Note that $\omega_n$ bounds a disk $\Omega_n$ in $S^3$ meeting $\hat{\beta}_n$ in $2g + 1$ points.
This is shown at the very bottom of \reffig{Braid}.
We deduce that $S^3 - \Lambda_n$ is a punctured disk bundle over the circle, with monodromy $\beta_n$. 

\begin{claim}
\label{Clm:Unknot}
Both $\omega_n$ and $\hat{\beta}_n$ are unknots in $S^3$.
\end{claim}

\begin{proof}
Since $\omega_n$ bounds the disk $\Omega_n$, it is an unknot.

Note that $\hat{\beta}_n$ is stabilised along its first strand.  
Destabilising has the effect of smoothing the crossing at $\sigma_{1}^{-1}$ and deleting the first strand.  
The factors $\Phi^n$ and $\Phi^{-n}$ now cancel,
leaving only the ($2g$--strand) braid closure $\hat{\Pi}$.
This is an iterated stabilisation,
proving the claim.
\end{proof}

Let $\gamma_n$ and $\delta_n$ be the augmentations of $\beta_n$, taken before and after the factor $\sigma_{1}^{-1}$, each containing the second, third, and fourth strands.
Again, see \reffig{Braid}.
Thus $\gamma_n$ and $\delta_n$ bound disks $\Gamma_n$ and $\Delta_n$ in $S^3$ each meeting $\hat{\beta}_n$ in three points.

Appealing to \refclm{Unknot}, the branched double cover of $S^3$ along $\hat{\beta}_n$ is again homeomorphic to $S^3$. 
Let $c_n$, $d_n$, and $w_n$ be the preimages of $\gamma_n$, $\delta_n$, and $\omega_n$, respectively; 
let $C_n$, $D_n$, and $W_n$ be the preimages of $\Gamma_n$, $\Delta_n$, and $\Omega_n$, respectively.
Since the disks meet $\hat{\beta}_n$ in an odd number of points, each of $c_n$, $d_n$, and $w_n$ is connected and equals the boundary of the corresponding surface $C_n$, $D_n$, and $W_n$.
An Euler characteristic calculation shows that $C_n$ and $D_n$ are homeomorphic to $S_{1, 1}$:
a torus with one boundary component; see, for example,~\cite[Figure~9.13]{FarbMargalit12}.
Similarly, $W_n$ is homeomorphic to $S_{g, 1}$:
a surface of genus $g$ with one boundary component.
We deduce that the knot complement $M_n = S^3 - w_n$ is an $S_{g, 1}$--bundle over $S^1$.
Thus $w_n$ is a genus $g$ fibred knot in $S^3$. 
This is the family of knots promised in \refthm{BigKnots}\refitm{Knot}.

By a result of Birman and Hilden~\cite[Theorem~9.2]{FarbMargalit12}, the monodromy of $M_n$ is a product of Dehn twists, one for each half-twist generator in $\calB_{2g+1}$. 
To fix notation, let $s_i$ be the Dehn twist in $S_{g,1}$ lifting $\sigma_i$. 
Lifting $\Pi$, $\Phi$, and $\beta_n$ in this way gives the following mapping classes.
\begin{align*}
P   &= s_{2g} \cdot s_{2g - 1}^{-1} \cdot s_{2g - 2} \cdot s_{2g - 3}^{-1} 
       \cdots s_{4} \cdot s_{3}^{-1} \cdot s_{2} \\ 
F   &= s_{2} \cdot s_{3}^{-1} \\ 
b_n &= P \cdot F^n \cdot s_{1}^{-1} \cdot F^{-n} 
\end{align*}
\noindent
Thus $b_n$ is the monodromy of $M_n = S^3 - w_n$.

\begin{claim}
\label{Clm:HypZero}
For all $g$, the knot complement $M_0$ is hyperbolic.  
The same holds for its longitudinal filling. 
\end{claim}

\begin{proof}
Since $n = 0$, the monodromy of $M_0$ simplifies to
\[
b_0 = s_{2g} \cdot s_{2g - 1}^{-1} \cdot s_{2g - 2} \cdot s_{2g - 3}^{-1} 
      \cdots s_4 \cdot s_{3}^{-1} \cdot s_{2} \cdot s_{1}^{-1}
\]
which is (cyclically) conjugate to 
\[
\left( s_{2g} \cdot s_{2g - 2} \cdots s_4 \cdot s_2 \right) 
\left( s_{2g - 1} \cdot s_{2g - 3} \cdots s_3 \cdot s_{1} \right)^{-1}.
\]
Let $\alpha_i \subset S_{g, 1}$ be the core curve of the Dehn twist $s_i$.
We isotope the curves $\alpha_i$ to intersect minimally. 
This done, $\alpha_i$ and $\alpha_j$ intersect (and then intersect in a single point) if and only if $|i - j| = 1$.
Thus, the union $\bigcup_i \alpha_i$ is connected.
Also, the complement of $\bigcup_i \alpha_i$ is a peripheral annulus. 
We now apply a criterion of
Thurston~\cite[Theorem~7]{thurston:geometry-dynamics} (see also Veech~\cite[pages~578--579]{Veech89}).  
Let $N$ be the matrix with $N_{ij} =  |\alpha_i \cap \alpha_j|$.
Let $\mu$ be the largest real eigenvalue of $NN^t$; 
note that $\mu$ is positive.
Let $A$ be the union of the $\alpha_i$ for $i$ even;
let $B$ be the union of the $\alpha_i$ for $i$ odd.
Let $T_A$ and $T_B$ be the corresponding multi-twists. 
The image of $T_A \cdot T_B^{-1}$ under Thurston's representation has trace $2 + \mu$,  
hence $b_0$ is pseudo-Anosov.
Appealing to Thurston's hyperbolisation theorem for mapping tori~\cite[Theorem~5.6]{thurston:survey}, we find that $M_0$ is hyperbolic.

If we fill, replacing $S_{g, 1}$ by $S_g$, then the complement of $\bigcup_i \alpha_i$ is a disk. 
Thus the same proof shows that the longitudinal filling of $M_0$ is hyperbolic.
This proves the claim.
\end{proof}


\begin{remark}
\label{Rem:TwoBridge}
By an observation of Gabai and Kazez~\cite[Proposition~2]{GabaiKazez}, the knots $w_0$ are two-bridge. 
In genus $g$, the continued fraction for $w_0$ is $[2, 2, \ldots, 2]$ with $2g$ terms.
This gives another proof that $M_0$ is hyperbolic.
\end{remark}

We now prove \refthm{BigKnots}\refitm{Hyp}.

\begin{claim}
\label{Clm:Hyp}
In fixed genus $g$, as $n$ tends to infinity, the knot complements $M_n$ are eventually hyperbolic,
with volumes tending to infinity.
\end{claim}

\begin{proof}
We fix a genus $g$ and reuse the notation above, 
taking $N = M_0 - (c_0 \cup d_0)$.
Since $M_0$ is hyperbolic (\refclm{HypZero}), since $c_0$ and $d_0$ lie in distinct fibres, and since $c_0$ and $d_0$ are not isotopic, we deduce that $N$ is hyperbolic by Thurston's hyperbolisation theorem~\cite[Theorem~2.3]{thurston:survey}.

Let $C, D \subset N$ be the images of $C_0$ and $D_0$ respectively.  
We cut the cusps off of $N$ and then cut along small open regular neighbourhoods of $C$ and $D$ to obtain a compact three-manifold $P$. 
Let $\rho \subset \bdy P$ be the remains of the cusp tori about $w_0$, $c_0$, and $d_0$. 
That is, $\rho$ consists of a torus, coming from $w_0$, and two \emph{paring annuli}, coming from $c_0$ and $d_0$.
The annular components of $\rho$ separate the genus two components of $\bdy P$ into one-holed tori $C^\pm$ and $D^\pm$ respectively.
These form the \emph{horizontal boundary} of $(P, \rho)$

Since $C$ and $D$ were contained in distinct fibres in $M_0$, the horizontal boundary of $(P, \rho)$ is incompressible in $(P, \rho)$.
Thus $(P, \rho)$ is a compact, oriented, irreducible, atoroidal three-manifold with incompressible horizontal boundary and with non-abelian fundamental group.
Since $c_0$ and $d_0$ are not homotopic in $M_0$, we deduce that $(P, \rho)$ is a \emph{pared manifold} as in~\cite[Section~2.4]{BrockMinskyNamaziSouto16}. 

For each horizontal boundary component $E$ of $(P, \rho)$ we now chose \emph{any} complete marking $\mu(E)$; 
see~\cite[Section~2.1]{BrockMinskyNamaziSouto16}.
Define $N_n = M_n - (c_n \cup d_n)$. 
Let $m(c_n)$ and $m(d_n)$ be the meridional slopes in the associated torus cusps of $N_n$. 
We note that $N_n$ is obtained from $(P, \rho)$ by gluing $C^+$ to $C^-$, and $D^-$ to $D^+$, using the $n^\mathrm{th}$ power of $F$, the monodromy of the figure-eight knot complement.
Since $F$ is pseudo-Anosov, we deduce that the gluings $N_n$ have \emph{$R$--bounded combinatorics} and increasing \emph{height} in the sense of~\cite[Section~2.12]{BrockMinskyNamaziSouto16}. 

We now apply a theorem of Brock, Minsky, Namazi, and Souto~\cite[Theorem 8.1]{BrockMinskyNamaziSouto16} to find that, for sufficiently large $n$, the gluing $N_n$ is hyperbolic and has a \emph{bilipschitz  model} $\MM_n$.
We deduce that the volumes of the manifolds $N_n$ tend to infinity coarsely linearly with $n$. 
Likewise, the lengths of the meridional slopes $m(c_n)$ and $m(d_n)$ tend to infinity. 
Applying a Dehn surgery result~\cite[Theorem~1.1]{fkp:volume}, we find that the manifolds $M_n$ are hyperbolic, and have volume tending to infinity coarsely linearly with $n$. 
This proves the claim.
\end{proof}

The proof of \refthm{BigKnots}\refitm{Long} is similar.  
Fix the genus $g$.
By \refclm{HypZero}, the longitudinal filling $M'$ of $M_0$ is hyperbolic. 
Taking $N'$ to be the longitudinal filling of $N$ along $w_0$, we deduce from the above that $N'$ is also hyperbolic.
We cut along the surfaces $C'$ and $D'$ (the images of $C_0$ and $D_0$) to obtain a pared manifold $(P', \rho')$.  
We now again apply the machinery of~\cite{BrockMinskyNamaziSouto16} and of~\cite{fkp:volume}.
This completes the proof of \refthm{BigKnots}.
\end{proof}

\begin{remark}\label{Rem:CLMProof}
There is another, related, proof of parts \refitm{Hyp} and \refitm{Long} of \refthm{BigKnots}.
One uses the work of Clay, Leininger, and Mangahas~\cite[Theorem~5.2]{ClayLeiningerMangahas12} to show that the monodromies $b_n$ have linearly growing subsurface projections to $C_n$ and $D_n$. 
Pairing this with the Masur-Minsky distance formula~\cite[Section~8]{masur-minsky:hierarchies} and results of Brock~\cite[Theorems~1.1 and~2.1]{brock:fibered}, we find that $M_n$ has linearly growing volume. 
\end{remark}


\begin{remark}[An enhancement of \refthm{BigKnots}]
Our knots $w_n$, produced via a branched double covering construction, are morally similar to a family of hyperbolic knots produced by Misev~\cite[Section~3]{Misev21}, using a plumbing construction.
However, in any fixed genus Misev's family has bounded hyperbolic volume.
Nevertheless, with a bit more work one may adapt the construction of \refthm{BigKnots} so that the family of knots has volume going to infinity while satisfying Misev's conclusions. In particular, for each $g \geq 2$, there is a family of knots $w_n$ with the following properties.
\begin{enumerate}
\item 
\label{Itm:NewFibred}
$S^3 - w_n$ is fibred over the circle, with fibre of genus $g$. 
\item 
\label{Itm:NewVolume}
$S^3 - w_n$ is hyperbolic, with volume at least $V$.
\item 
\label{Itm:NewLong}
Longitudinal surgery on $w_n$ is hyperbolic, with volume at least $V$.
\item 
\label{Itm:Quasipositive} 
$w_n$ is strongly quasipositive.
\item 
\label{Itm:Seifert} 
$w_n$ has the same Seifert form as the torus knot $T(2g + 1, 2)$.
\end{enumerate}

Following Rudolph~\cite{Rudolph98, Rudolph:PositiveLinks}, we call a knot $K$ \emph{strongly quasipositive} if it is the closure of a braid that lies in the monoid generated by all elements of the following form.
\[
(\sigma_j \sigma_{j-1} \dots \sigma_{i+1}) \sigma_i (\sigma_j \sigma_{j-1} \dots \sigma_{i+1})^{-1}
\]
By work of Livingston~\cite{Livingston:Computations}, such knots $K$ have $\tau(K) = g_4(K) = g_3(K)$: here $g_4$ is the four-ball genus, $g_3$ is the Seifert genus, and $\tau$ is the Ozsv\'ath--Szab\'o concordance invariant.
By a result of Hedden~\cite[Proposition 2.1]{Hedden:NotionsOfPositivity}, if $K$ is strongly quasipositive and fibred then it serves as the binding of an open book decomposition that supports the tight contact structure on $S^3$.
If $K$ has the same Seifert form as $T = T(2g + 1, 2)$ then $K$ has the same Alexander module and signature function as $T$.
We refer to Misev~\cite{Misev21} for further discussion and implications of these properties. 

To construct the sequence $w_n$ having properties \refitm{NewFibred}--\refitm{Seifert}, we modify the definitions of the braids $\Phi$, $\Pi$, and $\beta_n$ appearing in the proof of Theorem~\ref{Thm:BigKnots}.
We first require $\Phi$ to factor as a
product of (conjugates of) squares of full twists on sets of odd numbers of strands. 
We also require $\Phi$ to be a pseudo-Anosov element on the last $2g$ strands.
For instance, when $g = 2$, one can take
\begin{align*}
\Phi    &= \left( \sigma_3 \sigma_4^{2} \cdot (\sigma_2 \sigma_3)^6 \cdot  \sigma_4^{-2} \sigma_3^{-1}\right) \cdot 
           \left( \sigma_3^{-1} \sigma_4^{-2} \cdot (\sigma_2 \sigma_3)^6 \cdot  \sigma_4^{2} \sigma_3\right). \\
\intertext{We define $\Pi$ and $\beta_n$ as follows.}
\Pi     &= \sigma_{2g} \cdot \sigma_{2g - 1} \cdot \sigma_{2g - 2} \cdot \sigma_{2g - 3} 
           \cdot \cdot \cdot \sigma_4 \cdot \sigma_3 \cdot \sigma_{2} \\ 
\beta_n &= \Pi \cdot \Phi^n \cdot \sigma_1 \cdot \Phi^{-n} 
\end{align*}
Property~\refitm{NewFibred} now holds, for the same reason as above. 


The choices of $\Pi$ and $\beta_n$ ensure that the fibre of $w_n$ is a plumbing of a positive Hopf band onto the fibre of the torus link $T(2,2g)$.
The strong quasipositivity of $w_n$ now follows from the strong quasipositivity of $T(2,2g)$ by~\cite[Proposition~4.2]{Rudolph98}, yielding property \refitm{Quasipositive}.

Since $\Phi$ factors as product of squares of full twists (each about an odd number of strands), its lift $F$ factors as a product of Dehn twists along null-homologous curves. 
Since we are twisting along null-homologous curves, the Seifert form of $w_n$ is independent of $n$.
This yields property \refitm{Seifert}, as $w_0$ is the torus knot $T(2,2g+1)$.

Observe that $\beta_n$ factors as follows.
\[
\beta_n = (\Pi\, \sigma_1) \cdot ( \sigma_1^{-1} \Phi^n \sigma_1 \cdot \Phi^{-n} )
\]
The first term in parentheses is periodic; 
in fact, $(\Pi\, \sigma_1)^{2g+1}$ is a Dehn twist about the boundary curve $\omega_n$ 
(and thus trivial in the mapping class group). 
We now claim that $\beta_n$ is pseudo-Anosov. 
It suffices to prove that $\beta_n^{2g + 1}$ is pseudo-Anosov.
To see this, note that we can pull all copies of the first term $\Pi \sigma_1$ to the front by conjugating the copies of the second term, $\sigma_1^{-1} \Phi^n \sigma_1 \cdot \Phi^{-n}$.
Thus $\beta_n^{2g+1}$ is a Dehn twist about $\omega_n$, composed with a product of conjugates of the second term.
These are themselves a product of one conjugate of $\Phi^n$ and one conjugate of $\Phi^{-n}$.
Each conjugate is supported in a sub-disk containing $2g$ punctures, so every pair of supporting domains is neither disjoint nor nested.
Furthermore, each conjugate has large translation distance acting on the curve complex of its supporting domain. 
Thus~\cite[Theorem~6.1]{ClayLeiningerMangahas12} implies that $\beta_n^{2g+1}$ is pseudo-Anosov when $n$ is large.

Since $\beta_n$ is pseudo-Anosov, so is its lift $b_n$, hence $M_n$ is hyperbolic.
Since $\Phi$ is pseudo-Anosov on a sub-disk containing the last $2g$ strands, it follows that its lift $F$ is pseudo-Anosov on the branched double cover, a copy of $S_{g-1,2}$.
As in the previous paragraph, $b_n^{2g+1}$ factors as a product of conjugates of large powers of $F$ or $F^{-1}$.
Thus~\cite[Theorem~5.2]{ClayLeiningerMangahas12} implies that $b_n^{2g+1}$ has linearly growing translation distances in curve complexes of the corresponding supporting domains. 
Properties~\refitm{NewVolume} and~\refitm{NewLong} now follow exactly as in \refrem{CLMProof}.
\end{remark}

\begin{remark}[Other work]
\label{Rem:Other}
Baker~\cite[Theorem~4.1]{Baker08} finds among the Berge knots a sub-collection which have unbounded volume.  
However, as observed by Goda and Teragaito~\cite[page~502]{GodaTeragaito00}, all Berge knots are closures of positive (or negative) braids, hence there are only finitely many Berge knots of any given genus. Positivity also implies that the Berge knots are fibred.

Hirose, Kalfagianni, and Kin~\cite[Theorem~2]{HiroseKalfagianniKin22} give a construction of branched double covers (of any fixed closed, connected, oriented three-manifold $M$) that are fibred of genus $g \gg 0$ and hyperbolic with volume tending to infinity with $g$.  
They ask~\cite[Question~4]{HiroseKalfagianniKin22} whether, for every $M$, there is such a sequence with genus fixed and volume unbounded.

Our work gives a positive answer for $M = S^2 \cross S^1$, as follows. 
By \refclm{Unknot}, the longitudinal filling of $\omega_n \subset S^3$ is $M = S^2 \cross S^1$.
Let $\omega'_n \subset M$ be the core of the filling solid torus.
Then the longitudinal fillings of the knots $w_n$, as in \refthm{BigKnots}\refitm{Long}, are branched double covers of $M$ with branch locus $\hat{\beta}_n \cup \omega'_n$.
Finally, by using our \refthm{BigKnots}\refitm{Knot}\refitm{Hyp}, Hirose, Kalfagianni, and Kin give a positive answer to their Question~4 with  $M = S^3$ and genus even~\cite[Corollary~11]{HiroseKalfagianniKin22}.

Very recently, Oakley proved a version of our \refthm{BigKnots} for knots in arbitrary closed three-manifolds~\cite{Oakley}. 
Combined with~\cite[Theorem~10]{HiroseKalfagianniKin22}, this gives a complete positive answer to~\cite[Question~4]{HiroseKalfagianniKin22}.
\end{remark}

\subsection*{Acknowledgements} 
We thank Alan Reid for posing the problem that led to this note. 
We thank David Gabai for suggesting that we consider Murasugi sums.
We thank Sam Taylor for helpful conversations regarding \refclm{Hyp}.
We thank Effie Kalfagianni and Eiko Kin for explaining their work with Susumu Hirose~\cite{HiroseKalfagianniKin22}.

Baker was partially supported by Simons Foundation grants 523883 and 962034.
Futer was partially supported by NSF grant DMS--1907708.
Purcell was partially supported by Australian Research Council grant DP210103136.

\renewcommand{\UrlFont}{\tiny\ttfamily}
\renewcommand\hrefdefaultfont{\tiny\ttfamily}
\bibliographystyle{plainurl}
\bibliography{biblio}

\begin{thebibliography}{10}

\bibitem{Baker08}
Kenneth~L. Baker.
\newblock Surgery descriptions and volumes of {B}erge knots. {I}. {L}arge
  volume {B}erge knots.
\newblock {\em J. Knot Theory Ramifications}, 17(9):1077--1097, 2008.
\newblock \href {http://arxiv.org/abs/math/0509054}
  {\path{arXiv:math/0509054}}, \href
  {https://doi.org/10.1142/S0218216508006518}
  {\path{doi:10.1142/S0218216508006518}}.

\bibitem{BrockMinskyNamaziSouto16}
Jeffrey Brock, Yair Minsky, Hossein Namazi, and Juan Souto.
\newblock Bounded combinatorics and uniform models for hyperbolic 3-manifolds.
\newblock {\em J. Topol.}, 9(2):451--501, 2016.
\newblock \href {http://arxiv.org/abs/1312.2293} {\path{arXiv:1312.2293}},
  \href {https://doi.org/10.1112/jtopol/jtv043}
  {\path{doi:10.1112/jtopol/jtv043}}.

\bibitem{brock:fibered}
Jeffrey~F. Brock.
\newblock Weil--{P}etersson translation distance and volumes of mapping tori.
\newblock {\em Comm. Anal. Geom.}, 11(5):987--999, 2003.
\newblock \href {https://doi.org/10.4310/CAG.2003.v11.n5.a6}
  {\path{doi:10.4310/CAG.2003.v11.n5.a6}}.

\bibitem{ClayLeiningerMangahas12}
Matt~T. Clay, Christopher~J. Leininger, and Johanna Mangahas.
\newblock The geometry of right-angled {A}rtin subgroups of mapping class
  groups.
\newblock {\em Groups Geom. Dyn.}, 6(2):249--278, 2012.
\newblock \href {https://doi.org/10.4171/GGD/157} {\path{doi:10.4171/GGD/157}}.

\bibitem{FarbMargalit12}
Benson Farb and Dan Margalit.
\newblock {\em A primer on mapping class groups}, volume~49 of {\em Princeton
  Mathematical Series}.
\newblock Princeton University Press, Princeton, NJ, 2012.
\newblock \href {https://doi.org/10.1515/9781400839049}
  {\path{doi:10.1515/9781400839049}}.

\bibitem{fkp:volume}
David Futer, Efstratia Kalfagianni, and Jessica~S. Purcell.
\newblock Dehn filling, volume, and the {J}ones polynomial.
\newblock {\em J. Differential Geom.}, 78(3):429--464, 2008.
\newblock \href {http://arxiv.org/abs/math/0612138}
  {\path{arXiv:math/0612138}}, \href {https://doi.org/10.4310/jdg/1207834551}
  {\path{doi:10.4310/jdg/1207834551}}.

\bibitem{GabaiKazez}
David Gabai and William~H. Kazez.
\newblock Pseudo-{A}nosov maps and surgery on fibred {$2$}-bridge knots.
\newblock {\em Topology Appl.}, 37(1):93--100, 1990.
\newblock \href {https://doi.org/10.1016/0166-8641(90)90018-W}
  {\path{doi:10.1016/0166-8641(90)90018-W}}.

\bibitem{GodaTeragaito00}
Hiroshi Goda and Masakazu Teragaito.
\newblock Dehn surgeries on knots which yield lens spaces and genera of knots.
\newblock {\em Math. Proc. Cambridge Philos. Soc.}, 129(3):501--515, 2000.
\newblock \href {http://arxiv.org/abs/math/9902157}
  {\path{arXiv:math/9902157}}, \href
  {https://doi.org/10.1017/S0305004100004692}
  {\path{doi:10.1017/S0305004100004692}}.

\bibitem{Hedden:NotionsOfPositivity}
Matthew Hedden.
\newblock Notions of positivity and the {O}zsv\'{a}th-{S}zab\'{o} concordance
  invariant.
\newblock {\em J. Knot Theory Ramifications}, 19(5):617--629, 2010.
\newblock \href {https://doi.org/10.1142/S0218216510008017}
  {\path{doi:10.1142/S0218216510008017}}.

\bibitem{HiroseKalfagianniKin22}
Susumu Hirose, Efstratia Kalfagianni, and Eiko Kin.
\newblock Volumes of fibered 2-fold branched covers of 3-manifolds.
\newblock {\em Journal of Topology and Analysis}, To appear.
\newblock \href {http://arxiv.org/abs/2208.02977} {\path{arXiv:2208.02977}},
  \href {https://doi.org/10.1142/S1793525323500322}
  {\path{doi:10.1142/S1793525323500322}}.

\bibitem{Livingston:Computations}
Charles Livingston.
\newblock Computations of the {O}zsv\'{a}th-{S}zab\'{o} knot concordance
  invariant.
\newblock {\em Geom. Topol.}, 8:735--742, 2004.
\newblock \href {https://doi.org/10.2140/gt.2004.8.735}
  {\path{doi:10.2140/gt.2004.8.735}}.

\bibitem{masur-minsky:hierarchies}
Howard~A. Masur and Yair~N. Minsky.
\newblock Geometry of the complex of curves. {II}. {H}ierarchical structure.
\newblock {\em Geom. Funct. Anal.}, 10(4):902--974, 2000.
\newblock \href {https://doi.org/10.1007/PL00001643}
  {\path{doi:10.1007/PL00001643}}.

\bibitem{Misev21}
Filip Misev.
\newblock On families of fibred knots with equal {S}eifert forms.
\newblock {\em Comm. Anal. Geom.}, 29(2):465--482, 2021.
\newblock \href {http://arxiv.org/abs/1703.07632} {\path{arXiv:1703.07632}},
  \href {https://doi.org/10.4310/CAG.2021.v29.n2.a6}
  {\path{doi:10.4310/CAG.2021.v29.n2.a6}}.

\bibitem{Oakley}
J.~Robert Oakley.
\newblock Large volume fibered knots in 3-manifolds, 2023.
\newblock Preprint.
\newblock
  \url{https://drive.google.com/file/d/1WEf8zCGK0ZrZbog8z3wkr_dph-9SK26N/view}.

\bibitem{Rudolph98}
Lee Rudolph.
\newblock Quasipositive plumbing (constructions of quasipositive knots and
  links. {V}).
\newblock {\em Proc. Amer. Math. Soc.}, 126(1):257--267, 1998.
\newblock \href {https://doi.org/10.1090/S0002-9939-98-04407-4}
  {\path{doi:10.1090/S0002-9939-98-04407-4}}.

\bibitem{Rudolph:PositiveLinks}
Lee Rudolph.
\newblock Positive links are strongly quasipositive.
\newblock In {\em Proceedings of the {K}irbyfest ({B}erkeley, {CA}, 1998)},
  volume~2 of {\em Geom. Topol. Monogr.}, pages 555--562. Geom. Topol. Publ.,
  Coventry, 1999.
\newblock \href {https://doi.org/10.2140/gtm.1999.2.555}
  {\path{doi:10.2140/gtm.1999.2.555}}.

\bibitem{thurston:survey}
William~P. Thurston.
\newblock Three-dimensional manifolds, {K}leinian groups and hyperbolic
  geometry.
\newblock {\em Bull. Amer. Math. Soc. (N.S.)}, 6(3):357--381, 1982.
\newblock \href {https://doi.org/10.1090/S0273-0979-1982-15003-0}
  {\path{doi:10.1090/S0273-0979-1982-15003-0}}.

\bibitem{thurston:geometry-dynamics}
William~P. Thurston.
\newblock On the geometry and dynamics of diffeomorphisms of surfaces.
\newblock {\em Bull. Amer. Math. Soc. (N.S.)}, 19(2):417--431, 1988.
\newblock \href {https://doi.org/10.1090/S0273-0979-1988-15685-6}
  {\path{doi:10.1090/S0273-0979-1988-15685-6}}.

\bibitem{Veech89}
William~A. Veech.
\newblock Teichm\"{u}ller curves in moduli space, {E}isenstein series and an
  application to triangular billiards.
\newblock {\em Invent. Math.}, 97(3):553--583, 1989.
\newblock \href {https://doi.org/10.1007/BF01388890}
  {\path{doi:10.1007/BF01388890}}.

\end{thebibliography}

\end{document}